\documentclass[12pt,a4paper,twoside,english]{smfart}

\usepackage{lmodern,mdframed,xcolor}
\usepackage[utf8]{inputenc}
\usepackage[francais]{babel}

\usepackage[OT2,T1]{fontenc}

\usepackage{amssymb}
\usepackage{amsmath}
\usepackage{smfthm,mathabx}
\usepackage{enumerate}
\usepackage{amsthm}
\usepackage{amsfonts}
\usepackage{graphicx}
\usepackage[colorlinks=true,linkcolor=black,citecolor=black,urlcolor=black]{hyperref}
\usepackage{fancyhdr}
\usepackage[all,cmtip]{xy}
\usepackage{alltt}
\usepackage{footmisc}
\usepackage{smfthm}

\usepackage{tikz-cd}

\usepackage{calrsfs}

\usepackage{fullpage}

\xyoption{all}

\def\Q{{\mathbb Q}}
\def\Z{{\mathbb Z}}
\def\fq{{\mathbb F}}

\def\p{{\mathfrak p}}
\def\P{{\mathfrak P}}

\def\q{{\mathfrak q}}
\def\Qq{{\mathfrak Q}}

\def\O{{\mathcal O}}
\def\U{{\mathcal U}}

\def\Ll{{\mathcal L}}

\def\T{{\mathcal T}}

\def\1{{\bf 1}}

\def\Fr{{\rm Frob}}

\usepackage{calrsfs}

\newtheorem{Theorem}{Theorem}

\newtheorem*{Example}{Example}

\author{Donghyeok Lim}
\address{Institute of Mathematical Sciences, Ewha Womans University, Seoul 03760, Republic of Korea}
\email{donghyeokklim@gmail.com}

\author{Christian Maire}
\address{FEMTO-ST Institute, Universit\'e Franche-Comt\'e, CNRS,  15B avenue des Montboucons, 25000 Besan\c con, FRANCE} 
\email{christian.maire@univ-fcomte.fr}

\begin{document}

\title{On the analyticity of the maximal extension of  a number field with prescribed ramification and splitting}

%%%%%%%%%%%%%%%%%%%%%%%%%%%%%%%%%%%

\begin{abstract}  
We determine all the $p$-adic analytic groups that are realizable as Galois groups of the maximal pro-$p$ extensions of number fields with prescribed ramification and splitting 
under an assumption which allows us to move away from the Tame Fontaine-Mazur conjecture.
\end{abstract}

\thanks{This work has been started during a visiting position for the second author at Ewha Womans University, and finished during a visiting fellow at the Western Academic for Advanced Research (WAFAR) of Western  University;  CM thanks the  Department of Mathematics at Ewha University  and the WAFAR  for providing a beautiful research atmosphere. We would like to thank Bill Allombert for his help with PARI/GP, and C\'ecile Armana   for useful remarks. The first author was supported by the Core Research Institute Basic Science Research Program through the National Research Foundation of Korea (NRF) funded by the Ministry of Education (Grant No. 2019R1A6A1A11051177) and the Basic Science Research Program through the National Research Foundation of Korea (NRF) funded by the Ministry of Education (Grant No. NRF-2022R1I1A1A01071431). The second author was also partially supported by the EIPHI Graduate School (ANR-17-EURE-0002).}



\date{\today}

%%%%%%%%%%%%%%%%%%%%

%%%%%%%%%%%%%%%%%%%%%%%%%%%%%%%%%%%%%%%%%%%%%%%%%%%%

\maketitle 

%%%%%%%%%%%%%%%%%%%%%%%%%%%%%%%%%%%%%%%%%%%

\section*{Introduction}

For a number field $K$, its absolute Galois group $G_K$ is a fundamental object of study.
% in number theory.  
The last decades have shown that 
 (continuous) Galois representations  $$\rho : G_K \rightarrow Gl_n(\Q_p)$$ occupy a central position in arithmetic geometry, serving as a fundamental tool to provide a bridge between the geometric and arithmetic aspects of number theory. 
  A governing philosophy is the conjecture of Fontaine and Mazur \cite[Conjecture 1]{FM} that an irreducible $p$-adic Galois representation of~$G_{K}$ comes from geometric object if it is unramified outside a finite set of primes and its restrictions to the decomposition subgroups at primes above~$p$ are potentially semi-stable. A variation of the conjecture is the `Tame Fontaine-Mazur conjecture' that if $S$ is a finite set of non-$p$ primes of $K$, then a $p$-adic analytic quotient of $G_K$ that is unramified outside $S$ is always finite (\cite[Conjecture 5a]{FM}). 

Hence, it is natural to study which $p$-adic analytic groups can be realized as $G_{S}^{T}$ which is a Galois group naturally defined in terms of ramification and splitting of places of number fields.

 \medskip
 
 Let us be more precise.
 Let $S$ and $T$ be two finite and disjoint sets of places of $K$. 
 Let ${\bar K}_S^T$ (resp. $K_S^T$) be the maximal (resp. the maximal pro-$p$) extension of $K$ unramified outside $S$ 
 and completely decomposed at $T$. We put ${\bar G_{S}^T}:={\bar G}_{K,S}^T:=Gal({\bar K}_S^T/K)$ (resp. $G_S^T:=G_{K,S}^T:=Gal(K_S^T/K)$).

 In this paper, we are interested  in  Galois representations $\rho: {\bar G}_{K,S}^T \rightarrow Gl_n(\Q_p)$ with  $p$-closed image in ${\bar K}_S^T$,
{\it i.e.} such that $H^1(ker(\rho),\Z/p)=1$. 
 More precisely, we want to characterize the possible $\Q_p$-Lie algebra $\Ll(\rho)$  of  the image of~$\rho$. For example, when $K=\Q$, $S=T=\emptyset$ then $\Ll(\rho)=\{0\}$ for every Galois representation~$\rho$ of ${\bar G}_{K,\emptyset}^\emptyset$ since this Galois group is trivial.
 
 \medskip

 Every  {\it compact} $p$-adic analytic group   contains a torsion free  pro-$p$ group  as an open subgroup. Hence by base change,  one can assume that $S$ contains only finite places, and we can focus on $G_{K',S}^T$ for some finite extension $K'/K$ in ${\bar K}_S^T$.

 \medskip
 
 In general, this question is difficult because it is not easy to determine whether $G_{S}^{T}$ is not FAb group {\it i.e.} if its open subgroups have finite abelianization. If $G_{S}^{T}$ is FAb, then the problem of determining the analyticity of $G_{S}^{T}$ shares many difficulties with the (Tame) Fontaine-Mazur conjecture mentioned before. 
 
 Thus to make the study more accessible, we assume the following condition  $(C)$ $$1+\delta_S> |T|+r_1+r_2,$$
 where $\displaystyle{\delta_S}$ denotes the sum $\displaystyle{\sum_{\p\in S_p'} [K_\p:\Q_p]}$ for $S_p'=\{{\rm prime \ } \p \in S,  \p|p\}$, and $r_1$ (resp. $r_2$) is the number of  real (resp. complex) places of $K$. By the assumption $(C)$, the pro-$p$ group $G_S^T$ has $\Z_p$ as a quotient by class field theory (cf. \cite[Chapter III, Theorem 1.6]{gras}). In particular, we move away from the tame Fontaine-Mazur conjecture.

\medskip

 We first prove:
 
 \begin{Theorem}\label{TheoremA}
Assuming $(C)$, the pro-$p$ group $G_{S}^T$ is  a $p$-adic analytic group if and only if, it is virtually isomorphic to:
\begin{enumerate}
\item[$(i)$] $\Z_p$, or
 \item[$(ii)$] $\Z_p \rtimes \Z_p$ (noncommutative), or
 \item[$(iii)$] $\Z_p \times \Z_p$.
\end{enumerate}
Moreover, we have $\delta_S=r_1+r_2+|T|$. 
 \end{Theorem}

Observe that when $G_S^T \simeq \Z_p$, then $G_S^T$ is potentially  of local type.
 Here, {\it potentially of local type} means that there exists a prime $\p|p$ of $K_S^T$ above a prime in $S$ such that the decomposition subgroup of $G_S^T$ at $\p$ is open. This notion was studied by Wingberg in \cite{Wingberg_local}.
 We will  observe that if  $\zeta_p\in K$, then $G_S^T$ is also potentially of local type in the cases $(ii)$ and $(iii)$.

 \medskip
 
 Back to the original question: let  $\rho : {\bar G}_{K,S}^T \rightarrow Gl_n(\Q_p)$ be a  Galois representation with $p$-closed image in ${\bar K}_S^T$. Then in $(i)$ (resp. in $(iii)$) the Lie algebra $\Ll(\rho)$ is the abelian $\Q_p$-algebra of dimension $1$ (resp. of dimension  $2$).
In $(ii)$, $\Ll(\rho)$ is the noncommutative  Lie algebra of dimension $2$; $\Ll(\rho)$ can be generated by  $x$ and $y$ satisfying the relation $[x,y]=x$.

 \medskip
 
 It is easy to produce examples of type $(i)$ (namely when $K=\Q$ and $S=S_p$, the set of primes of $K$ that are  $p$-adic). The examples of type $(ii)$ were studied by Wingberg \cite{Wingberg_local} when $S_p \subset S$ and $T=\emptyset$. However, no example of type $(iii)$ was known.
 As the second result, one obtains:

\begin{Theorem}\label{TheoremB}
Let $p$ be an odd prime. There is a number field $K$ and a finite set $T$  of primes of $K$ such that $G_{K,S_p}^{T} \simeq \Z_p\times \Z_p$.
The set $T$ is given by the Chebotarev density theorem.
\end{Theorem}

We remark that $\Z_p\times \Z_p$ cannot be realized as $G_{S}$ when $S$ contains $S_p$ because $G_{S}$ has Euler-Poincar\'{e} characteristic $-r_2$ whereas  $\Z_p\times \Z_p$ has Euler-Poincar\'{e} characteristic $0$. Hence, if $G_{S}$ is isomorphic to $\Z_p\times \Z_p$, then $K$ is totally real. In that case, the $\Z_p$-rank of $G_{S}$ is~$1$ by Leopoldt conjecture. 

\smallskip

By a numerical computation, we also find an example for $p=2$ for which $G_S^T\simeq \Z_p\times \Z_p$.

\begin{Example}
 Take $K=\Q(\zeta_8)$. 
Let $\p$ (resp. $\q$) be the prime ideal $(2+\zeta_8+2\zeta_{8}^2)$ (resp. $(6-\zeta_8+6\zeta_{8}^2)$) of $K$ above $7$ (resp. $71$). Then, $G_{K,S_2}^{\{\p,\q\}} \simeq \Z_p\times \Z_p$. 
\end{Example}

\medskip

The paper contains three sections. In Section \ref{section_arithmetic},  we recall basic facts about pro-$p$ groups and arithmetic in  pro-$p$-extensions of a number field.
 In Section \ref{section_proofs},   we prove Theorems \ref{TheoremA} and \ref{TheoremB}. 
 The last section is devoted to some remarks. In particular, the proof of Theorem~\ref{TheoremB} allows us to compute a lower bound for the number of  sets $T=\{\p,\q\}$ of primes of $K$ such that $N\p, N\q \leq X$, {$G_{K,S_p}^{T} \simeq \Z_p\times \Z_p$} which holds for generic pairs $(K,p)$ of an imaginary biquadratic field $K$ and an odd prime $p$ under the recent conjecture of Gras on $p$-rationality of number fields.

 \medskip
 
 All calculations were performed using PARI/GP  \cite{pari}.

 \medskip
 
 \medskip
 
{\bf Notations.} 
Throughout this article $p$ is a prime number. 

$\bullet$ If $M$ is a finitely generated $\Z_p$-module, set $d_p M:=dim_{\fq_p} M/M^p$, $M[p]:=\{m\in M, pm=0\}$,  and  $\mathrm{rk}_{\Z_p}M={\rm dim}_{\Q_p} \ \Q_p \otimes_{\Z_p} M$. 

$\bullet$ Let $G$ be a {finitely generated} pro-$p$ group. Set $G^{ab}:=G/[G,G]$, $G^{p,el}:=G^{ab}/(G^{ab})^p$, and $d_p G:= dim_{\fq_p}G^{p,el}$. For $n\geq 1$, $(G_n)$ denotes the Zassenhaus filtration of $G$ (cf. \cite[Chapter 7]{Koch}).

%%%%%%%%%%%%%%%%%%%%%%%%%%%%%%%%%%%%%%%%%%%%%%%%%%%%%%%%%%%%%%%%%%%%%%%%%%%ù
%%%%%%%%%%%%%%%%%%%%%%%%%%%%%%%%%%%%%%%%%%%%%%%%%%%%%%%%%%%%%%%%%%%%%%%%%%%ù
%%%%%%%%%%%%%%%%%%%%%%%%%%%%%%%%%%%%%%%%%%%%%%%%%%%%%%%%%%%%%%%%%%%%%%%%%%%ù

\section{Generalities on pro-$p$ groups and Galois groups with restricted ramification}\label{section_arithmetic}

In this section, we briefly recall basic facts  that are necessary in this paper.

\subsection{The partial Euler-Poincar\'{e} characteristic of pro-$p$ groups}

Let $G$ be a finitely generated pro-$p$ group. 
 Recall that the cohomological dimension $cd(G)$ of a pro-$p$ group $G$ is defined to be the smallest integer $k$ such that $H^{k}(G,\Z/p) \neq 0$ and $H^{k+1}(G, \Z/p)=0$. 
 
 Suppose that the groups  $H^i(G,\Z/p)$ are finite for $i=1,\cdots, n$. Then the $n$-th partial Euler-Poincar\'{e} characteristic $\chi_n(G)$ is defined to be
 $$\chi_n(G)=\sum_{i=0}^n (-1)^i d_pH^i(G,\Z/p). $$
 The cohomological dimension of a pro-$p$ group can be studied by the partial Euler-Poincar\'{e} characteristic according to the following theorem of Schmidt \cite{Schmidt0}.

\begin{prop}\label{SchmidtEuler}
Let $G$ be a pro-$p$ group such that $H^{i}(G,\Z/p)$ is finite for $0 \leq i \leq n$. Suppose that there is an integer $N$ such that $(-1)^n\chi_n(U)+N \geq (-1)^n(G:U)\chi_n(G)$ for all open subgroups $U$ of $G$. Then either~$G$ is finite or $cd(G) \leq n$.
\end{prop}

We will apply Proposition \ref{SchmidtEuler} for $n=2$. In that case, $\chi_2(G)$ is intimately related to the $\Z_p$-rank of $G^{ab}$. Let us write $$G^{ab} \simeq \Z_p^t \oplus \T,$$ where $\T$ is the torsion subgroup of $G^{ab}$.
Recall the following well-known result.
\begin{prop} \label{prop2}
 One has $$\chi_2(G)=1+d_pH_2(G,\Z_p)-t.$$
 Moreover, the  group $G$ is free pro-$p$ if and only if $H^2(G,\Q/\Z)=0$ and $\T=1$.
\end{prop}

\begin{proof}
By taking the $G$-homology of the exact sequence $0 \longrightarrow \Z_p \stackrel{p}{\longrightarrow} \Z_p \longrightarrow \Z/p \longrightarrow 0$, we obtain the following exact sequence
 
 $$ 0 \longrightarrow H_2(G,\Z_p)/p \longrightarrow H_2(G,\Z/p) \longrightarrow H_1(G,\Z_p)[p] \longrightarrow 0. $$
 Both claims follow from the isomorphism $H_1(G,\Z_p)\simeq G^{ab}$ and the duality between cohomology and homology groups.
\end{proof}

\subsection{On the pro-$p$ groups $G_S^T$} \label{section_GST}

Let $K$ be a number field, and $S,T$ be two  finite disjoint sets of primes of $K$.  In this work, we will assume that $S$ consists only of finite places.
Set
\begin{itemize}
\item[$\bullet$] $S_p$ : the set of primes of $K$ above $p$, $S_p'=S\cap S_p$, and $\delta_S:=\delta_{S_p'}:=\underset{\p\in S_p'}{\sum}[K_\p:\Q_p]$,
    \item[$\bullet$] $E^T:=E_K^T$ the \textit{pro-$p$ completion} of the  group of $T$-units of  $K$,
  \item[$\bullet$]  $K_\p$ the completion of $K$ at $\p|p$, $U_\p$ the group of units of $K_\p$,
\item[$\bullet$]  $\displaystyle{\U_{S}:=\prod_{\p \in S_p'}\U_\p}$, 
and $\displaystyle{\U_\p:= \lim_{\stackrel{\longleftarrow}{n}} U_\p/U_\p^{p^n}}$  the pro-$p$ completion of  $U_\p$,
\item[$\bullet$] $\delta:=\delta_{K,p}=1$ (resp. $\delta_\p=1$) if $\zeta_p \in K$ (resp. $\zeta_p \in K_\p$), $0$ otherwise,
\item[$\bullet$] For every $\p \in S \backslash S_p$, we assume that $\delta_\p=1$,
\item[$\bullet$] $\varphi:=\varphi_S^T : E^T  \to \U_S$   the diagonal embedding of $E^T$ into $\U_S$,
\item[$\bullet$] $V_S^T=\{x \in K^\times \, | \, v_\p(x)\equiv 0 \mod p \ \forall \p \notin T\ \& \ x\in K_\p^{\times \, p} \ \forall \p \in S\}$ where $v_{\mathfrak{p}}(x)$ denotes the discrete valuation of $x$ at $\mathfrak{p}$, 
\item[$\bullet$]  $K_S^T/K$  the maximal pro-$p$ extension of $K$ unramified outside $S$ and completely decomposed at $T$;  $G_S^T:=G_{K,S}^T:=Gal(K_S^T/K)$,
\item[$\bullet$] $\T_S$ the torsion part of $G_S^{ab}$ (here $T=\emptyset$),
\item[$\bullet$] If $L/K$ is a finite extension, by abuse we still denote  $S:=S_L:=\{\P|\p, \p \in S\}$. 
\end{itemize}

\medskip

The pro-$p$ group $G_{S}^T$ is well-known to be finitely presented.  More precisely, one has  $$d_p G_{S}^T = 1+ \underset{\mathfrak{p} \in S}{\sum} \delta_{\mathfrak{p}} -\delta   + d_p V_S^T/K^{\times \, p}+ \delta_S  -(r_1+r_2+|T|) $$
and
$$ d_p H^2(G_{S}^T,\Z/p) \leq \underset{\mathfrak{p} \in S}{\sum} \delta_{\mathfrak{p}} - \delta + d_p V_S^T/K^{\times \, p} + \theta,$$
where $\theta$ is equal to $1$ if $\zeta_p\in K$ and $S= \emptyset$, and zero in all other cases.
(See \cite[Chapter X, Theorem 10.7.10]{NSW}.)

\medskip

Therefore, we have the inequality $$ \chi_2(G_S^T) \leq \theta+ r_1+r_2+|T|-\delta_S. $$
In particular under the assumption $(C)$, one has
\begin{equation}\label{relations} \chi_2(G_S^T) \leq  0. \end{equation}

From the above explicit formulae of Shafarevich and Koch, we also have the following fact on the Schur multiplicator $H_2(G_S^T,\Z_p)$ of $G_{S}^T$ (cf. Lemme 3.1 of \cite{Maire-JTNB}). 

\begin{lemm} \label{lemm_bounded_H2}
The $p$-rank of $H_{2}(G_{S}^{T}, \Z_p)$ is bounded above by $\theta + \mathrm{rk}_{\Z_p}ker (\varphi_{S}^{T})$.
\end{lemm}
\begin{proof}
By Proposition \ref{prop2} and the formulae of Shafarevich and Koch, we have the inequality $$d_p H_2(G_{S}^T,\Z_p) = \chi_2(G_S^T) -1 + \mathrm{rk}_{\Z_p}G_{S}^{T,ab} \leq -1-\delta_S+r_1+r_2+|T|+\theta + \mathrm{rk}_{\Z_p}G_S^{T,ab}. $$
The claim follows from the equality $\mathrm{rk}_{\Z_p} G_{S}^{T,ab} = \delta_S- (r_1+r_2+|T|-1) + \mathrm{rk}_{\Z_p}ker(\varphi_{S}^T)$.
\end{proof}

We study $G_{S_p}^{T}$ by considering it as a quotient of $G_{S_p}$ by the (normal subgroup generated by the) Frobenius automorphisms at the primes of $K_{S_p}$ above $T$. The key idea of the proof of Theorem~\ref{TheoremB} is as follows: for any finite quotient $G$ of $G_{S_p}$, we can use Chebotarev density theorem to find some primes whose Frobenius restrict to any prescribed elements of $G$.
Let us recall relatively strong properties of $G_{S_p}$. See \cite[Proposition 8.3.18, Corollary 8.7.5, and Corollary 10.4.8]{NSW}.

\begin{theo} \label{theo_GS} Suppose that $S$ contains $S_p$ and assume that $K$ totally imaginary if $p=2$.
 The pro-$p$ group $G_{S}$ has cohomological dimension $1$ or $2$. Moreover, we have
 $\chi_2(G_S)=-r_2$.
\end{theo}
To make our strategy of using Chebotarev density theorem as easy as possible, it is nice to consider the case when $G_{S}$ is free pro-$p$.
   Observe that if $G_S$ is free pro-$p$ then there is no tame ramification in $K_S/K$.

  \begin{prop}\label{prop_free} Let $K$ be a number field and $S$ a finite set of places of $K$. If $ker(\varphi_{S})=1$ and $\T_{S}=1$, then $G_{S}$ is free pro-$p$. The converse is also true if $S=S_p$. Furthermore, we have $d_p G_S = 1+\delta_S-(r_1+r_2)$.
   \end{prop}

   \begin{proof}
   This is a consequence of Proposition \ref{prop2} and Lemma \ref{lemm_bounded_H2}.
   Moreover when $S=S_p$, we have $\chi_2(G_{S_p})=-r_2$ (see Theorem \ref{theo_GS}) which implies $$d_p H_2(G_{S_p},\Z_p)=\mathrm{rk}_{\Z_p}ker (\varphi_{S}).$$
Hence in this case, if $G_{S}$ is free pro-$p$, then we have $ker(\varphi_{S})=1$.
   \end{proof}

Under the Leopoldt conjecture, $G_{S_p}$ is free pro-$p$  if and only if $\T_{S_p}=1$. Even though the freeness of $G_{S_p}$ seems to be strong, it is believed to be a common phenomenon. In particular, we have the following conjecture.

%%%%%%%%%%%%%%%%%%%%%%%%%%%%%%%%%%%%%%%%%%%%%%%%%%%%%%%%%%%%%%%%%%%%%%%%%%%%%
%%%%%%%%%%%%%%%%%%%%%%%%%%%%%%%%%%%%%%%%%%%%%%%%%%%%%%%%%%%%%%%%%%%%%%%%%%%%%
%%%%%%%%%%%%%%%%%%%%%%%%%%%%%%%%%%%%%%%%%%%%%%%%%%%%%%%%%%%%%%%%%%%%%%%%%%%%%
%%%%%%%%%%%%%%%%%%%%%%%%%%%%%%%%%%%%%%%

\begin{conj}[Gras \cite{Gras-CJM}] \label{conjecture_gras} Given a number field $K$, then $\T_{S_p}=1$ for $p\gg 0$. 
\end{conj}

To be complete, let us recall that when $G_{S_p}$ is  free pro-$p$, then $K$ is said to be {\it $p$-rational} (\cite{Movahhedi-phd}).

%%%%%%%%%%%%%%%%%%%%%%%%%%%%%%%%%%%%%%%%%%%%%%%%%%%%%%%%%%%%%%%%%%%%%%%%%%%%%
%%%%%%%%%%%%%%%%%%%%%%%%%%%%%%%%%%%%%%%%%%%%%%%%%%%%%%%%%%%%%%%%%%%%%%%%%%%%%
%%%%%%%%%%%%%%%%%%%%%%%%%%%%%%%%%%%%%%%%%%%%%%%%%%%%%%%%%%%%%%%%%%%%%%%%%%%%%
%%%%%%%%%%%%%%%%%%%%%%%%%%%%%%%%%%%%%%%

\smallskip

We finish this subsection with a well-known fact on the $\Z_p$-rank of $G_S^T$ \cite{gras}. By class field theory, the $\Z_p$-rank of the abelianization of $G_{K,S}^{T}$ is equal to the $\Z_p$-rank of the cokernel of the diagonal map $\varphi: E^{T} \to \U_S$. If $K/\Q$ is Galois, then considering Galois actions of $Gal(K/\Q)$ is useful as in the following lemma which will be important in Theorem B.

\begin{lemm}\label{biquadraticT}
Let $K/\Q$ be an imaginary biquadratic field. Let $K^+$ be its real quadratic subfield. Let $T$ be a {\it non-empty} finite   set of non-$p$ primes of $K$. If the primes of $T$ are fixed by $Gal(K/K^+)$, then the $\Z_p$-rank of $G_{K,S_p}^{T}$ is $2$.
\end{lemm}

\begin{proof} Let $\1$ be the trivial $\Q_p$-character of $Gal(K/K^+)$, and $\chi$ be the nontrivial character.
The character of the $\Q_p$-representation $\mathcal{U}_{S_p} \otimes_{\Z_p}\Q_p$ is equal to $\1+\1 + 2\chi$. On the other hand, the character of the $T$-units is $(|T|+1) \1$. Hence, the character of the image of $\varphi_{S_p}^T$ is contained in the isotypic component at $\1$. 
It is precisely $\1 + \1$ because for any non-$p$ prime $\p$ of $K$, the $\Z_p$-rank of $G_{K,S_p}^{\{\p\}}$ is strictly smaller than $G_{K,S_{p}}$; $\p$ does not split completely in the cyclotomic $\Z_p$-extension of $K$.
\end{proof}

 \section{Proof of the main results} \label{section_proofs}
 
In this section, we prove the main theorems of this work. They completely give answer to the question of the realizability of analytic groups as $G_{K,S}^{T}$ under the assumption $(C)$.
 
 \subsection{Proof of Theorem \ref{TheoremA}}
 
 \begin{theo}\label{maintheo}
Assuming $(C)$, the pro-$p$ group $G_S^T$ is   a $p$-adic analytic group if and only if it is virtually isomorphic to one of
 $\Z_p$, $\Z_p \rtimes \Z_p$ (noncommutative), and
  $\Z_p \times \Z_p$.
In particular, we have
$\delta_S=r_1+r_2+|T|$.
 \end{theo}
 
 \begin{proof} The proof combines the argument of Proposition 3.3 of \cite{Maire-JTNB} and properties of $p$-adic analytic groups (\cite{DSMN}). Suppose that $G_S^T$ is $p$-adic analytic, then the $p$-rank of open subgroups $U$ of $G_{S}^T$ are uniformly bounded. Hence, the $\Z_p$-ranks of $U$ are also uniformly bounded. If $L$ is the subfield of $K_{S}^T$ fixed by an open subgroup $U$, then we have the following equality
\begin{equation}\label{Zp-rank}
rk_{\Z_p} U^{ab}=[L:K](\delta_S-(r_1+r_2+|T|))+1+rk_{\Z_p} ker(\varphi_{L,S}^T).
\end{equation}
 Hence, if $G_{S}^T$ is $p$-adic analytic, then necessarily we have
 \begin{enumerate}
  \item[$(a)$] $\delta_S=r_1+r_2+|T|$ and,
  \item[$(b)$] the rank of the kernel of $\varphi_S^T$ is bounded along $K_S^T/K$.
 \end{enumerate}
By Lemma \ref{lemm_bounded_H2}, $(b)$ implies that the $p$-rank of $H_{2}(U,\Z_p)$ is uniformly bounded for all open subgroups $U$ of $G_{S}^T$. By Proposition \ref{prop2}, $|\chi_{2}(U)|$ for open subgroups $U$ of $G_{S}^T$ are uniformly bounded. Since $\chi_{2}(G)$ is non-positive by the assumption $(C)$ (see (\ref{relations})), for some sufficiently large integer $N$, we have $\chi_{2}(U) + N \geq (G:U)\chi_{2}(G)$ for all $U$. Therefore, either $G_{S}^T$ is finite or $cd(G_{S}^T) \leq 2$ by Proposition \ref{SchmidtEuler}. By the assumption $(C)$, the pro-$p$ group $G_S^T$  is never finite.
One concludes thanks to  the classification of the $p$-adic analytic groups of dimension $2$.
\end{proof}

Observe that when $G_S^T \simeq \Z_p \rtimes \Z_p$, whether it is commutative or not is related to the behavior of $\ker (\varphi_{L,S}^T)$ for number fields $L$ in $K_S^T/K$.
 
 \begin{prop}
  Suppose that $G_S^T$ is a uniform pro-$p$ group of dimension $2$.
  Then $ \mathrm{rk}_{\Z_p} \ker(\varphi_{L,S}^T) \in \{0, 1\}$ is constant along  $K_S^T/K$. Moreover, $ \mathrm{rk}_{\Z_p}\ker(\varphi_{L,S}^T) =1$ if and only if 
  $G_S^T \simeq \Z_p \times \Z_p$.
 \end{prop}

 \begin{proof}
The claim follows from the classification of uniform pro-$p$ groups of rank $2$, the formula $(\ref{Zp-rank})$, and the conclusion (a) in the proof of Theorem \ref{TheoremA}.
 \end{proof}

 \subsection{Proof of Theorem \ref{TheoremB}}
Now, let us prove that $\Z_p\times \Z_p$ can be realized as a Galois group $G_{K,S_p}^{T}$ for a number field $K$ and a finite set $T$ of primes of $K$. We use the $p$-rational number fields.

Take $p$ odd.
Let $K$ be an imaginary biquadratic $p$-rational field. The existence of such a number field is already known from the works \cite{BenMov,Koperecz}. 
We will take $S=S_p$. Then,~$T$ is necessarily equal to $\{\p,\q\}$ for some non-$p$ primes of $K$ by the conclusion $(a)$ of Theorem~\ref{TheoremA}. Suppose that $\p$ is a non-$p$ prime of $K$ whose Frobenius automorphism $\Fr_\p$ in $G_{S_p}:=G_{K,S_p}$ represents a non-trivial element in the vector space $(G_{S_p})^{p,el} \simeq \mathbb{F}_p^{3}$. Then we have the following easy lemma.

\begin{lemm} \label{lemm_free}  The pro-$p$ group $G_{S_p}^{\{\p\}}$ is free pro-$p$  on $2$ generators.
\end{lemm} 
%\begin{proof}  Observe that $d_pG_{S_p}^{\{\p\}}=2$. Let   $R \subset G_{S_p}$ be the minimal normal closed subgroup of $G_{S_p}$ containing $\Fr_\p$.  The Hochschild-Serre spectral sequence of the exact sequence  $1 \longrightarrow R \longrightarrow G_{S_p} \longrightarrow G_{S_p}^{\{\p\}} \longrightarrow 1$ allows us to obtain that $H^2(G_{S_p}^{\{\p\}},\Z/p)=0$.
%\end{proof}

If $\q$ is a non-$p$ prime of $K$ distinct from $\p$, then for the set $T=\{\p, \q\}$, $G_{S_p}^{T}$ is a one relator pro-$p$ group of   rank $2$ unless it is isomorphic to $\Z_p$.

A main difficulty in proving $G_{S_p}^{T} \simeq \Z_p\times \Z_p$ is that we cannot apply Chebotarev density theorem in an infinite Galois extension. However, if $G_{S_p}^T$ is already known to be a one relator pro-$p$ group, then we can use Chebotarev density theorem for $\q$ in a finite quotient of $G_{S_p}^{\{\p\}}$ to guarantee that $G_{S_p}^{T}$ is a Demushkin pro-$p$ group.
 
 \begin{prop}\label{Zp^2criteria} Let $S$ and $T$ be disjoint and finite sets of primes of $K$ such that $\delta_S=r_1+r_2+|T|$. Suppose that $(G_{K,S}^T)^{ab}\simeq \Z_p\times \Z_p$. Let $K_1,\cdots, K_{p+1}$ be the $p+1$ degree-$p$ extensions of $K$ in $K_S^T/K$. Then $G_{K,S}^T \simeq \Z_p\times \Z_p$ if and only if $$d_p G_{K_1,S}^T=\cdots =d_p G_{K_{p+1},S}^T=2.$$
 \end{prop}

\begin{proof}
 One direction is obvious.
 
 Suppose now that $d_p G_{K_1,S}^T=\cdots =d_p G_{K_{p+1},S}^T=2$. Then by Schreier's formula, the pro-$p$ group $G_{K,S}^T$ is not free. 
 Moreover by hypothesis and (\ref{relations}), one has $d_p H^2(G_S^T,\fq_p) \leq 1$. Therefore, $G_{K,S}^T$ is a pro-$p$-group with one relator. By the assumption on $d_p G_{K_i,S}^T$ and \cite[Chapter III, Theorem 3.9.15]{NSW}, $G_{K,S}^T$ is a Demushkin group (on two generators). We are done since Demushkin pro-$p$ groups are uniquely determined by their abelianizations.
\end{proof}

Proposition \ref{Zp^2criteria} provides us a simple criterion to check numerically whether $G_{K,S}^T$ is Demushkin with existing algorithms.  
It also implies that whether $G_{K,S}^T$ is Demushkin is determined by the class of the Frobenius at $\q$ in the quotient of $G_{K,S}^{\{\p\}}$ by the Frattini subgroup of the Frattini subgroup of $G_{K,S}^{\{\p\}}$ which is \textit{finite}. We can understand this also in the following way.

\begin{prop}\label{quadraticrelation} Let $F$ be a free pro-$p$  group of generator rank $2$ with generators $x,y \in F$.
Let $r$ be an element of $F_2=F^p(F,F)$. Set $R$ to be the smallest normal closed subgroup of $F$ generated by $r$.
Then the quotient group $F/R$ is Demushkin if and only if $r$ is congruent to $[x,y]^{i}$ modulo $F_3$ for an $i \in \Z$ prime to $p$. 
\end{prop}

\begin{proof} The group $G=F/R$ is an one-relator pro-$p$ group of rank $2$. Observe that  $i \in (\Z/p)^\times$ if and only if the
cup-product $H^{1}(G,\mathbb{F}_{p}) \times H^{1}(G, \mathbb{F}_{p}) \to H^{2}(G,\mathbb{F}_{p})$ is non-trivial (cf. \cite[Chapter III, Proposition 3.9.13 (ii)]{NSW}, \cite[Theorem 7.23]{Koch}). Since $H^{1}(G,\mathbb{F}_p)$ has $p$-rank $2$, this is equivalent to the non-degeneracy of the cup-product.
\end{proof}

Theorem \ref{TheoremB} is implied by the following theorem.

\begin{theo}\label{p-rationalTheoremB}
Let $p$ be an odd prime  and let $K$ be an imaginary biquadratic $p$-rational field. Then there are infinitely many sets $T$ of primes of $K$ with $|T|=2$ such that $G_{K,S_p}^T$ is isomorphic to $\Z_p\times \Z_p$. 
\end{theo}

\begin{proof} 
Let $K$ be an imaginary biquadratic field that is $p$-rational \cite{BenMov,Koperecz}: the pro-$p$ group  $G_{K,S_p}$ is free pro-$p$  of rank $3$ (see Proposition \ref{prop_free}). Let $K^+$ be the real quadratic subfield of $K$; we put $\Delta=Gal(K/K^+)$ and $s$ the generator of $\Delta$. Let $\p$ be a prime of $K^+$ which is inert in $K^{+}_\infty K/K^+$ where $K_\infty^+$ is the  cyclotomic $\Z_p$-extension of $K^+$. Then $G_{K,S_p}^{\{\p\}}$ is  free pro-$p$ of rank $2$  by Lemma \ref{lemm_free}.
We  remark that  $(G_{K,S_p}^{\{\p\}})^{p,el}$ is isomorphic to $\mathbb{F}_p^{-} \oplus \mathbb{F}_p^{-}$ as $\mathbb{F}_p[\Delta]$-modules.

\medskip

Set $F:=G_{K,S_p}^{\{\p\}}$.
Let $x,y$ be a system of minimal topological generators of $F$. By \cite{Herfort-Ribes} the elements $x$ and $y$ can be chosen such that $x^s=x^{-1}$ and $y^s=y^{-1}$, where $x^s$ and $y^s$ denote the image of the conjugate action of $s$ on $x$ and $y$ respectively.
Denote by $(F_n)$ the Zassenhaus filtration of $F$.

\medskip

Set $T=\{\p,\q\}$. If $G_{K,S_p}^T$ is not isomorphic to $\Z_p$, then it is a one relator pro-$p$ group of rank $2$. Suppose that a finite prime $\q$ of $K$ satisfies the following two conditions:
\begin{enumerate}
  \item[$(a)$] $\q$ is lying over a prime of $K^+$ that is inert in $K/K^+$, and
  \item[$(b)$] the Frobenius automorphism in $F$ corresponding to a prime of $K_{S_p}^{\{\p\}}$ above $\q$ is congruent to $[x,y]^{i}$ modulo $F_3$ for some $i \in \Z$ that is coprime to $p$.
 \end{enumerate}
The condition $(b)$ implies that $G_{K,S_p}^T$ is Demushkin by Proposition \ref{quadraticrelation}. On the other hand, the $\Z_p$-rank of $(G_{K,S_p}^T)^{ab}$ is $2$ by $(a)$ and Lemma \ref{biquadraticT}. Hence, $G_{K,S_p}^T$ is isomorphic to $\Z_p\times \Z_p$.

\medskip 

Let us now prove the existence of such a prime $\q$. By the choice of $\p$, the extension $K_{S_p}^{\{\p\}}/K^+$ is Galois with Galois group isomorphic to $F \rtimes \Delta$ by the Schur-Zassenhaus theorem.
Let $M$ be the subfield of $K_{S_p}^{\{\p\}}$ fixed by $F_3$: the field $M$ is still Galois over $K^+$ and we have $Gal(M/K^+) \simeq  Gal(M/K) \rtimes \Delta$. Let us use $x,y$ to denote also their images in $Gal(M/K)$. 
Take $j \in \Z$ coprime to $p$. By Chebotarev density theorem, there is a prime $\q$ of $K^+$ such that the Frobenius automorphism $\Fr_{\Qq}$ in $Gal(M/K^+)$ at a prime $\Qq$ of $M$ above $\q$ is in the conjugacy class of $([x,y]^j,s) \in Gal(M/K^+)$. The restriction of $\Fr_{\Qq}$ to $K$ is $s \in \Delta$. Therefore, the prime $\q$ is inert in $K/K^+$. By  definition, $(\Fr_{\Qq})^2$ is the Frobenius automorphism of $Gal(M/K)$ at $\Qq$, and this Frobenius automorphism is equal to $([x,y]^{j(1+s)},1)=(([x,y][x^{-1},y^{-1}])^j,1) \in Gal(M/K)$.
An easy computation in the Magnus algebra of $F$ shows that $[x^{-1},y^{-1}] \equiv [x,y]$ modulo $F_3$.  Therefore, $(\Fr_{\Qq})^2$ satisfies the condition $(b)$. If we take $T=\{\p, \q\}$, then $G_{K,S_p}^{T}$ is isomorphic to $\Z_p\times \Z_p$.
\end{proof}

\begin{rema} Take $K/K^+$, $S$ and $T$ as before. Let $T_0$ be any set of primes $\p$ of $K^+$ inert in $K/K^+$: by \cite[Corollary 3.2]{Hajir-Maire-mu} these primes split completely in $K_S^T/K$. Then $G_S^{T\cup T_0}\simeq \Z_p\times \Z_p$. Hence, the assumption $(C)$ is not absolutely necessary. 
\end{rema}

\begin{exem}
Let $K$ be the number field $\Q(\theta)$ with $\theta$ a fixed root of $X^4+1$. Then $K$ is $2$-rational.
Let $\p_{7}$ (resp. $\p_{71}$) be the prime ideal $(2+\theta+2\theta^2)$ (resp. $(6-\theta+6\theta^2)$) of $K$ above $7$ (resp. $71$). Set $T=\{\p_7, \p_{71}\}$. The pro-$2$ group $G_{S_2}^T$ is defined by two generators and one relation.

We can check that $K_{1}=K(\sqrt{\theta})$, $K_{2}=K(\sqrt{\theta^3-\theta^2+1})$, and $K_{3}=K(\sqrt{-\theta^3+\theta-1})$  are the three quadratic extensions in $K_{S_2}^T/K$. A computation shows that $d_2 G_{K_i,S_2}^T=2$, for $i=1,2,3$. By Proposition~\ref{Zp^2criteria}, one deduce that the group $G_{S_2}^T$  is Demushkin. Since $(G_{S_2}^{T})^{ab} \simeq \Z_p\times \Z_p$, one concludes that $G_{S_2}^T \simeq \Z_p\times \Z_p$.
\end{exem}

\section{Some remarks}

\subsection{When $G_S^T$ is of local type}

 When $G_S^T \simeq \Z_p$, $G_S^T$ must be potentially of local type by the finiteness of the class group.

 \begin{prop}
 Suppose that the assumption $(C)$ is true and $G_S^T$ is isomorphic to $\Z_p \rtimes \Z_p$ (possibly commutative). Then for every finite Galois extension $L/K$ in $K_S^T/K$, one has $$-\delta_{L,p} + \sum_{\p \in S_L} \delta_{\p} + d_p V_{L,S}^T/L^{\times \ p}=1.$$
 In particular when $\zeta_p \in K$ then $|S| \leq 2 $ for all $L$. Hence, $G_S^T$ is of local type for $p>2$ and potentially of local type for $p=2$.
 \end{prop}

 \begin{proof}
  Given a finite Galois extension $L/K$ in $K_S^T/K$, the equality in the statement follows from $$d_p H^1(G_{L,S}^T,\Z/p)=2$$ and the conclusion $(a)$ in the proof of Theorem \ref{TheoremA}.
   When $\zeta_p \in K$, one must have $|S_L \cap S_p|\leq 2 $. Hence, for any $p$-adic prime $\p_0 \in S$ the decomposition group $G_{\p_0}$ of $\p_0$ in $K_S^T/K$ is exactly $G_S^T$ if $p$ is odd, and an open subgroup of $G_S^T$ if $p=2$. 
 \end{proof}
 
 \subsection{Quantities on $T$ with $G_{S_p}^T \simeq \Z_p\times \Z_p$}
 
Let $K$ be an imaginary biquadratic number field and $p$ an odd prime. Let $K^+$ be the real quadratic subfield of $K$; we put $\Delta=Gal(K/K^+)$. 
For $X \geq 2$, set $$A(X)= \big \{\{\p,\q\} \ {\rm  a \ set \ of \  primes \ of \ } K \ \big |  \ N_{K/\Q}\p , N_{K/\Q}\q \leq X, \ G_{K,S_p}^{\{\p,\q\}} \simeq \Z_p\times \Z_p \big \}.$$

By the conjecture of Gras, the proof of Theorem B gives us the following statistics of $|A(X)|$ for generic couples $(K,p)$.

\begin{prop}\label{lowerbound}
Let $p$ be an odd prime and $K$ an imaginary biquadratic $p$-rational number field. Then as $X \to \infty$
$$|A(X)| \geq c_p  \ \frac{X}{(\log X)^2},$$
where $c_p$ is some constant depending on $p$.
\end{prop}

It is easy to understand that $|A(X)|$ is very small compared to $(X/\log X)^2$ because the primes $\p$ and $\q$ have residue class degrees larger than $1$.

\begin{proof}
We will compute a lower bound for the number of $T$ in the proof of Theorem~\ref{TheoremB}.
Let $K'$ be  the first layer of the cyclotomic $\Z_p$-extension of $K$. Let $K_2:=K_p^{p,el}$ be the maximal elementary abelian $p$-extension of $K$ that is unramified outside $p$ ; $Gal(K_2/K)\simeq G_{K,S_p}^{p,el}$. Following the proof of Theorem \ref{TheoremB}, observe that the prime $\p$ is inert in $K'/K$.

Let $M$ be as in the proof of Theorem \ref{TheoremB}. Set $M'=MK'$. Observe that $Gal(M'/K)\simeq Gal(K'/K)\times Gal(M/K)$. Let us choose a prime $\q$ such that its Frobenius automorphism in $Gal(M'/K)$ has trivial component at $Gal(K'/K)$ and its restriction to $M$ is as in the proof of Theorem \ref{TheoremB}. By this choice, $\q$ splits completely in $K_2/K$, and there is no symmetry between $\p$ and $\q$.

\medskip

By using the argument of the proof of Theorem \ref{TheoremB}, we can check that the number of $\p$ that is inert in $K/K^+$ such that $G_{K,S_p}^{\{\p\}}$ is free pro-$p$  of rank $2$ and $N_{K/\Q}\p \leq X$ is asymptotically
\begin{equation}\label{prime1}
\frac{(p-1)\sqrt{X}}{p \log X}
\end{equation}
as $X \to \infty$ by applying Chebotarev density theorem in $K_2/K^+$. For each aforementioned $\p$, let $N_\p(X)$ be the number of primes $\q$ of $K$ with $N_{K/\Q}\q \leq X$ such that $\q$ splits in $K'$, $\q$ is equal to its conjugate over $K^+$, and its Frobenius automorphism is in the conjugacy class of $([x,y]^j,s) \in Gal(M/K^+)$ for some $j$ coprime to $p$. If $F$ is free pro-$p$ of rank~$2$, then we have $(F:F_3)=p^3$ and $(M':K^+)=2p^4$.  Therefore, $N_\p(X)$ is asymptotically %$G_{K,p}^{\{\p,\q\}} \simeq \Z_p^2$ satisfies
\begin{equation}\label{prime2}
  \frac{(p-1)\sqrt{X}}{p^4\log X}
\end{equation}
as $X \to \infty$ by applying Chebotarev density theorem in $M'/K^+$. We conclude thanks to the proof of Theorem \ref{TheoremB}, $(\ref{prime1})$, and $(\ref{prime2})$. In particular, $c_p$ is bounded below by $(p-1)^2/p^5$. \end{proof}

\subsection{Other pro-$p$ groups of the form $G_S^T$}

Using a similar argument, we can study the statistics of the sets $T=\{ \p, \q\}$ of primes of imaginary biquadratic $p$-rational fields $K$ such that $G_{S_p}^{T}$ is isomorphic to $\Z_p$ or $\Z_p \rtimes \Z_p$ (noncommutative). The number of the sets $T=\{ \p, \q \}$ with $G_{S_p}^{T} \simeq \Z_p$ and $N_{K/\Q}\p , N_{K/\Q}\q \leq X$ is asymptotically equal to $$\frac{(p^3-1)(p^2-1)}{2p^5} \frac{X^2}{(log X)^2}$$ as $X \to \infty$. On the other hand, the number of the sets $T=\{ \p, \q \}$ with $G_{S_p}^{T} \simeq \Z_p \rtimes \Z_p$ (noncommutative) and $N_{K/\Q}\p, N_{K/\Q}\q \leq X$ is asymptotically equal to $$\frac{(p^3-1)(p^2-1)}{2p^7} \frac{X^2}{(log X)^2} $$ as $X \to \infty$.
The statistics for $\Z_p$ come from the fact that a quotient of the free pro-$p$ group of rank $3$ by two elements is isomorphic to $\Z_p$ if and only if the classes of two elements in the maximal elementary abelian quotient are linearly independent. The statistics for $\Z_p \rtimes \Z_p$ (noncommutative) can be computed by noting that at least one of $\p$ and $\q$ does not belong to the Frattini subgroup of $G_{S_p}$ and applying Proposition \ref{quadraticrelation}. We can disregard the possibility of $G_{S_p}^T \simeq \Z_p^2$ because the number in Section 3.2 is negligible. 

%For a recent result on finding quotient of $G_{S}^T$ of various cohomological dimensions, we refer the readers to \cite{Hamza}.

%\begin{rema}
%\textcolor{red}{The lower bound obtained in Proposition \ref{lowerbound} can be larger than expected because the Frobenius automorphisms in the statement of the Chebotarev density theorem are uniformly distributed. Meanwhile, for a random element $r$ of a pro-$p$ free group $F$ of rank $2$, the quotient $F/R$ of $F$ by the normal subgroup $R$ generated by $r$ is rarely isomorphic to $\Z_p^2$. Indeed, if we try to approach this problem by using Chebotarev Density Theorem on Galois extensions \textit{over $F$}, we get the zero density as the answer.} 
%\end{rema}

%%%%%%%%%%%%%%%%%%%%%%%%%%%%%%%%%%%%%%%%%%%%%%%%%%%%%%%%%%%%%%%%%
%%%%%%%%%%%%%%%%%%%%%%%%%%%%%%%%%%%%%%%%%%%%%%%%%%%%%%%%%%%%%%%%%
%%%%%%%%%%%%%%%%%%%%%%%%%%%%%%%%%%%%%%%%%%%%%%%%%%%%%%%%%%%%%%%%%
%%%%%%%%%%%%%%%%%%%%%%%%%%%%%%%%%%%%%%%%%%%%%%%%%%%%%%%%%%%%%%%%%

%%%%%%%%%%%%%%%%%%%%%%%%%%%%%%%%%%%%%%%%%%%%%%%%%%%%%%%%%%%%%%%%%
%%%%%%%%%%%%%%%%%%%%%%%%%%%%%%%%%%%%%%%%%%%%%%%%%%%%%%%%%%%%%%%%%
%%%%%%%%%%%%%%%%%%%%%%%%%%%%%%%%%%%%%%%%%%%%%%%%%%%%%%%%%%%%%%%%%
%%%%%%%%%%%%%%%%%%%%%%%%%%%%%%%%%%%%%%%%%%%%%%%%%%%%%%%%%%%%%%%%%

\end{document}